\documentclass[11pt]{amsart}
\usepackage{amsmath,amstext,amssymb,amsopn,amsthm,mathrsfs}
\usepackage[top=4cm, bottom=3cm, left=3cm, right=3cm]{geometry}
\usepackage[shortlabels]{enumitem}
\usepackage{amsfonts}
\usepackage{amsthm}
\usepackage[sc]{mathpazo}
\usepackage[T1]{fontenc}
\usepackage[makeroom]{cancel}
\usepackage{microtype}

     \usepackage{calc}  

\usepackage{color}

\usepackage[dvips]{graphicx} 

\newtheorem{theorem}{Theorem}[section]

\newtheorem{lemma}{Lemma}[section]
\newtheorem{proposition}{Proposition}[section]

\numberwithin{equation}{section}

\begin{document}

\title[New real variable methods in $H$ summability of Fourier series]{New real variable methods in $H$ summability of Fourier series.}

\author{Calixto P. Calder\'on}
\address{Department of Mathematics, University of Illinois at Chicago, IL, 60607, USA.}
\email{[Calixto P. Calder\'on]cpc@uic.edu}
\author{A. Susana Cor\'e}
\address{Department of Mathematics, Illinois State University, Normal IL, 61761, USA.}
\email{[A. Susana Cor\'e] bacore@ilstu.edu}
\author{Wilfredo O. Urbina}
\address{
 Department of Mathematical and Actuarial Sciences, Roosevelt University  Chicago, IL, 60605, USA.}
\email{[Wilfredo Urbina]wurbinaromero@roosevelt.edu}

\begin{abstract} In this paper we shall be concerned with $H_\alpha$ summability, for  $0  < \alpha \leq 2$ of the Fourier series of arbitrary $L^1([-\pi,\pi]) $ functions. The method to be employed is a refinement of the real variable method introduced by Marcinkiewicz in \cite{Marcin1}. 

\end{abstract}
\maketitle
\begin{center}
{\em Dedicated to the memory of  A. Eduardo Gatto}
\end{center}
\section{Introduction}
Let $f$ be a function in $L^1([-\pi,\pi])$, denote by $S_n(f,\cdot)$ the partial sum of order $n$ of the Fourier series of $f$,
\begin{equation}
S_n(f,x) = \sum_{|k| \leq n} c_n e^{-inx}, \quad x \in [-\pi,\pi].
\end{equation}

We say that $f$ is $H_2$ summable at $x$ if there exists a number $s$ such that,
$$ \frac{ 1 }{ n } \sum_{k=1}^n |S_k(f,x) - f(x) |^2
         \longrightarrow 0   \quad \text{a.e.}$$
This can be extended easily to $\alpha>0$; i.e.  we say that its Fourier series is $H_\alpha$ summable to some $f(x)$ or that it is a strongly $\alpha$-summable to sum $f(x)$,  if\begin{equation}
  \frac{ 1 }{ n } \sum_{k=1}^n |S_k(f,x) - f(x) |^\alpha 
         \longrightarrow 0   \quad \text{a.e.}
\end{equation}
Historically the problem goes back to H. Hardy and J. H. Littlewood in \cite{Hardy-Little1}. There the problem is restricted to $H_2$ summability  of $L^2([-\pi,\pi]) $ functions ( i.e.  $\alpha = 2$ and 
  $ f \in L^2([-\pi,\pi]) $, see also  T. Carleman \cite{carle}.\\
  
  In 1935 Hardy and Littlewood proved the case $ \alpha > 0$ and 
  $  f \in L^p, \text{ with }, \;1 < p < \infty$  and posed the problem of whether ``any arbitrary periodic function in $L^1([-\pi,\pi])$  is $H_2$ summable a. e. in $[-\pi,\pi]$.  The answer to this question came only on January of 1939, when  J.  Marcinkiewicz  presented his remarkable result \cite{Marcin1}, developing a real variable method to establish it.\\
  
  Finally the case of $H_\alpha$ summability a. e. for $\alpha> 2$ and $ f \in L^1([-\pi,\pi]) $
 was proved by A. Zygmund in 1941, \cite{zy0} using complex methods.  In view of the negative results concerning convergence a.e. of the Fourier series of functions in $L^1([-\pi,\pi]) $ the $H_\alpha$ summability acquires a special meaning. \\
 
  In this paper we shall be concerned with $H_\alpha$ summability, for  $0  < \alpha \leq 2$ of Fourier series for arbitrary $L^1([-\pi,\pi]) $ functions. The methods used here are a refinement of the real variable method by Marcinkieicz in \cite{Marcin1}, and could be applied also to the case $\alpha >2$. Nevertheless this requires a modification of the Marcinkiewicz function and a change of kernel function (to bee defined later).\\
  
 \section{Preliminaries}
  Consider the following  maximal operator
 \begin{eqnarray}
 (\sigma_\alpha^\ast f)(x) &=& \sup_{n>0} 
                   \left[ \frac{ |S_1(f,x)|^\alpha 
                           + |S_2(f,x) |^\alpha + \ldots + |S_n(f,x) |^\alpha}{n } \right]^{1/\alpha}\\
\nonumber                           &=&  \sup_{n} \sqrt[\alpha]{  \frac{1}{n}\sum_{k=1}^n |  S_k(f,x) |^\alpha},    
\end{eqnarray}
where, as before, $S_k(f,\cdot)$  stands for the $k-$th partial sum of the Fourier series of $ f \in L^1([-\pi,\pi]) $ and $0 < \alpha \leq 2.$\\

Also, let us consider for $f \in L^1([-\pi,\pi])$ the non-centered Hardy-Littlewood function, namely,
\begin{equation}
                 f^\ast (x) = \sup_{\substack{
                                    I \supset \{x\} }
                                   }
                                   \frac{ 1 }{ |I| } 
                                   \int_I
                                  |f(y)| \ dy
         \end{equation}
where $I$ is taken to be an open interval, containing $x$. Observe that the set
$$
   F =   \Big\{ x \, : \, f^\ast (x) \leq \lambda \Big\}  $$
 is a closed set, and  the set
 $$  G =   \Big\{ x \, : \, f^\ast (x)  >   \lambda \Big\},$$
  is an open set. \\
  
  The class $A_1$ of weights is defined using the non-centered Hardy-Littlewood function, $f^\ast$,we say 
$\omega \in A_1$ if the inequality 
\begin{equation}
\omega(\Big\{ x \, : \, f^\ast (x) > \lambda \Big\}  ) \leq \frac{C}{\lambda}  \int_{-\pi}^\pi  |f(x)| \omega(x) \ dx,
\end{equation}
holds true for any $f \in L^1([-\pi,\pi])$,   $f \geq 0$  $C$ depends only on $\omega$.

A well known result gives a characterization of the weights in the case $(-\infty, \infty)$, see Stein \cite{stw};
a positive weight  $w \geq 0$ belongs to the class $A_1$ if and only if 
 \begin{equation}
\omega^* (x) \leq C \omega(x).\\
\end{equation}

In order to prove  the problem of $H_2$ summability for $L^1$ functions  Marcinkiewicz proved that
$\sigma_2^*$ is finite a.e. and he refined that to $H_2$ summability. Moreover, it can be proved, see Stein \cite{st0},
\begin{equation}
|\big\{x: (\sigma_2^\ast f)(x) > \lambda \big\} | \leq \frac{C}{\lambda} \|f\|_1,
\end{equation}
$C$ depends only on $\omega$.

  In order to tackle that problem Marcinkiewicz  introduced the so called {\em Marcinkiewicz function}. If $F$ is a perfect set and $G=F^c$ its complement,  if $d(x,F)= \inf_{z \in F} |x-z|,$ denotes the distance from $x$ to $F$; then he defined,
  \begin{equation}\label{MarcinFunct}
\mathcal{F}(x) = \int_G \frac{1}{(x-y)^2} d(x, F) \,dy
\end{equation}
  which is finite a.e.  for $x \in F$. 
  
  The function $\mathcal{F}(x)$ has important implications in the $L^1$ theory of singular integrals. In particular, in 1966,  L. Carleson in his famous $L^2$ theorem uses a variation of this function, Carleson function is denoted as $\Delta$ in his article. Also Zygmund shows the realtion between $\Delta(x)$ and $\mathcal{F}(x)$, see \cite{zy1}.
  
By Kolmogorov's counterexample in 1926, we know that there exists a function $f \in L^1([-\pi,\pi])$ such that $S_n(f, \cdot)$ diverges a.e., thus the maximal function,
 $$ f^{\ast \ast} (x)= \sup_n |S_n(f,x)|$$ 
 can not be weak type $(1,1)$, see Zygmund \cite{zy2}, Vol II.
 
  Now, consider firstly functions supported on an interval of length  $\pi/8$  , centered  at the origin. For those functions the maximal function $\sigma_\alpha^\ast f$ satisfies the inequality
    $$( \sigma_\alpha^\ast f)(x) \leq C f^\ast(x), $$
  for $x$ such that $d(0,x)\geq \pi/4$, where $C$ is a constant depending on $\pi$ only. The above inequality is consequence of
  $$ \int_{-\pi}^\pi |f(t)| dt \leq 2 \pi f^\ast(x),$$
  and the estimate 
  $$ \big|\frac{\sin(n+1/2)u}{\sin u/2}\big| \leq \frac{1}{|u|} +\frac{1}{2},$$ 
  see A. Zygmund \cite{zy2}, Vol I pages 50--51.
   In what follows we shall introduce a majorization of the kernel used by Marcinkiewicz in \cite{Marcin1}. 
   
 Returning to the case of a general function $f$, such such a function can be decomposed into a sum of pieces, each supported on an interval of length $\pi/8$. By a shift each one of the pieces is moved to the origin, thus each piece can be studied as if it were supported on an interval of length $\pi/8$ centered at the origin.\\
  
  The method to be employed in this paper is a refinement of the real variable method introduced by Marcinkiewicz in \cite{Marcin1}. 
  
  \section{Main results}
  
  The main results obtained in this paper are the following,

  \begin{theorem} Given $0  < \alpha \leq 2$, $ f \in L^1([-\pi,\pi]) $  and $\sigma_\alpha^\ast f$ as above, then
   \begin{equation}
                      | \Big\{ x \, : \, (\sigma_\alpha^\ast f)(x) > \lambda \Big\}| \leq   \frac{ C_\alpha }{\lambda  }  \|f\|_1,
\end{equation}
where $C_\alpha$ is a constant that depends  only on the $\alpha$ (but not on $f$).

 Morevoer,   if $\omega$ is an $A_1$-weight, then
 \begin{equation}
   \omega \Big(  |\big\{x:  (\sigma_\alpha^\ast f)(x) > \lambda \big\}| \Big) \leq 
           \frac{ C_\omega }{ \lambda }   \int_{-\pi}^\pi  |f(x)| \omega(x) \ dx,
\end{equation}
 $C_\omega$ is a constant that depends  only on the weight (but not on $f$).

\end{theorem}
 As a particular case, if  $ \omega (x) \equiv 1$ we have the Lebesgue's measure case.\\

  The following result was given by J. Marcinkiewicz in \cite{Marcin1} (see Lemma 3),
  
  \begin{lemma}
             Let $P$ be a perfect set
             (i.e., a set without isolated points), 
             $\Delta_\nu$ a sequence of contiguous segments,
             $\varphi (x)$ the function of period $2\pi$ equal to zero
             in $P$ and to $|\Delta_\nu|$
             for $x\in \Delta_\nu$. 
             We have  almost everywhere in $P$
             \begin{equation}
             \int_{-\pi}^{\pi} \frac{ \varphi (t+x) }{ t^2 }  \ dt < \infty.
             \end{equation}
           \end{lemma} 
 
For the proof  see \cite{Marcin1}. We will give an alternative proof of this result.
 
\subsection{Whitney type decomposition.} 

We start considering an special type of covering for open sets in $\mathbb{R}$ which has the same type of building principle than the Whitney decomposition in  $\mathbb{R}^n$, see Stein \cite{st1}, 
\begin{lemma}\label{lemacub}
  Let $G$ be an open set, and consider its decomposition into open disjoint intervals  $\big\{ J_k\big\} $  ( i.e. $G$ can be written as  $ G = \bigcup_k J_k$, and $J_k \cap J_l = \emptyset$ for $k \neq l$) such that $J_k$ are its connected components and the end points of $J_k$ are in $F= G^c $. Then it is possible to find a countable refinement   $\big\{ I_j\big\} $ such that, 
         \begin{enumerate}
\item [i)]   whenever $I_j$ and $I_{j'}$ are not adjacent, i.e. 
         $ \bar{I}_j \cap \bar{I}_{j'}  = \emptyset $
         then, for a suitable $C$

     \begin{equation*}
         \begin{split}
                 d( I_j , I_{j'}) & \ge C |I_j| \;\mbox{and}\\
                       d( I_j , I_{j'})  & \ge C |I_{j'}|.
         \end{split}
         \end{equation*}
         $C$ can be chosen to be bigger or equal to $1/2$.\\
  \item[ii)] For any $j$, $d(I_j,F) = |I_j|.$
\end{enumerate}

 \end{lemma}
 \begin{proof}
 
         We have
    $ G = \bigcup_k J_k  = \bigcup_k (a_k, b_k)$ with $a_k, b_k \in F=G^c$.
    For each fixed $k$ we do the following:
    \begin{enumerate}[leftmargin=1.7cm, rightmargin=1cm, label=\textbullet ]
            \item 
                    We decompose $J_k$ into three subintervals
                    $J_{k,1}, \ J_{k,2} , \ J_{k,3}$
                    having equal length, thus
                    \begin{equation*}
                            J_k = \bigcup_{i=1}^3 J_{k,i} \quad \text{and} \quad
                            d ( J_{k,i}, F)  = |J_{k,i}|
                    \end{equation*}
                    $J_{k,2}$ is selected to be a closed interval and $J_{k,1} , \ J_{k,3}$ are open intervals.
            \item
                    $J_{k,2}$, the central subinterval,
                    will be an element of the new refinement $\big\{ I_j\big\} $.
            \item
                    Each of the side open intervals, 
                    $J_{k,1}$ and $J_{k,3}$
                    are broken up into two subintervals of the same length such that, the ones that are adjacent to the central interval $J_{k,2}$, are taken such that the left one $J_{k,1,2}$ is closed on the left and open on the right and the right one $J_{k,3,1}$ is open on the left and closed to the right and therefore
                    \begin{equation*}
                    \begin{split}
                            d( J_{k,1,2}, F ) &= | J_{k,1,i} |
                                         \qquad i=1, 2\\
                            d ( J_{k,3,1}, F ) &= | J_{k,3,i} |.
                                         \qquad i=1, 2
                    \end{split}
                    \end{equation*}
                    
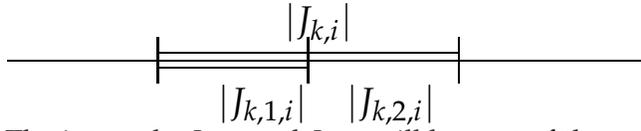
\setlength{\unitlength}{4144sp}%
\begingroup\makeatletter\ifx\SetFigFont\undefined%
\gdef\SetFigFont#1#2#3#4#5{%
  \reset@font\fontsize{#1}{#2pt}%
  \fontfamily{#3}\fontseries{#4}\fontshape{#5}%
  \selectfont}%
\fi\endgroup%
\begin{picture}(3849,780)(-11,-391)
{\color[rgb]{0,0,0}\thinlines
}%
{\color[rgb]{0,0,0}\put(901, 74){\line( 0,-1){270}}
}%
{\color[rgb]{0,0,0}\put(1801, 74){\line( 0,-1){270}}
}%
{\color[rgb]{0,0,0}\put(2701, 74){\line( 0,-1){270}}
}%
{\color[rgb]{0,0,0}\put(901,-106){\line( 1, 0){900}}
}%
{\color[rgb]{0,0,0}\put(901,-16){\line( 1, 0){1800}}
}%
{\color[rgb]{0,0,0}\put(  1,-61){\line( 1, 0){3825}}
}%
\put(1261,-376){\makebox(0,0)[lb]{\smash{{\SetFigFont{14}{16.8}{\familydefault}{\mddefault}{\updefault}{\color[rgb]{0,0,0}$|J_{k,1,i}|$}%
}}}}
\put(2031,-376){\makebox(0,0)[lb]{\smash{{\SetFigFont{14}{16.8}{\familydefault}{\mddefault}{\updefault}{\color[rgb]{0,0,0}$|J_{k,2,i}|$}%
}}}}
\put(1660,109){\makebox(0,0)[lb]{\smash{{\SetFigFont{14}{16.8}{\familydefault}{\mddefault}{\updefault}{\color[rgb]{0,0,0}$ |J_{k,i}|$}%
}}}}
\end{picture}%

    The intervals,  $J_{k,1,2}$ and $J_{k,3,1}$ 
    will be part of the new refinement $\big\{ I_j\big\} $.
    
    \item
             As before, the remaining side open intervals  $J_{k,1,1}$ and $J_{k,3,2}$ are broken up into two subintervals of the same length, the one that is adjacent to the interval $J_{k,1,2}$, is taken such that is closed on the left and open on the right and the one that is adjacent to the interval $J_{k,3,1}$ is open on the left and closed to the right and they  will be part of the new refinement $\big\{ I_j\big\} $.
             
   \item Iterating this argument  over and over again and doing the same process for each $J_k$ of the original decomposition of $G$ we obtain  a sequence of intervals $\big\{ I_j\big\} $,
             such that $ d(I_j, F) = |I_j| $.
    \end{enumerate}
    It is important to  note that  if  $ {I}_j$ and ${I}_{j'}$ are not adjacent, i.e. 
         $ \bar{I}_j \cap \bar{I}_{j'}  = \emptyset ,$
         then there will be among them
         at least one subinterval satisfying the construction conditions,
         and therefore they satisfy       
         \begin{eqnarray*}
                 d( I_j , I_{j'}) & \ge C | I_j|\\
                               & \ge C |I_{j'}|
                \end{eqnarray*}

 \end{proof}

 \subsection{Consequences of this Whitney type decomposition}
 Now, given $f \in L^1[-\pi,\pi])$, $f \geq 0$ and $\lambda>0$ consider the set
  \begin{equation*}
     G= \{x: f^\ast(x) = \sup \frac{ 1 }{ |I| }\int_I f(t) \, dt > \lambda \},
  \end{equation*}
then $G$ is an open set, and consider the Whitney type decomposition for $G$, $\big\{ I_j\big\}$, as above, i.e. $G = \cup_{j=1}^\infty I_j$.
   We take its average
             \begin{equation*}
                     \frac{1  }{|I_j|  }   \int_{I_j} f \, dx
                     \leq \frac{1 }{|I_j| } \int_{\widetilde{I}_j} f \, dt
             \end{equation*}
             where $\widetilde{I}_j$ has been obtained from $I_j$ by expanding it 2 times, i.e. $|\widetilde{I}_j| \geq 2 |I_j|. $
             If  $I_j$ is one of the central subintervals of the original decomposition  $J_{k,2},$ one of the adjacent
             subintervals is also included. 
             If  $I_j$ is not the central subintervals, we choose another subinterval, adjacent to the central one.
             In this way we have
                     $ \widetilde{I}_j =  I_j  \cup I $
                     and therefore $|\widetilde{I}_j| = 2 |I_j|$
              \begin{align*}
                \frac{1  }{|I_j| } \int_{I_j} f \, dx
                   & \leq  \frac{1  }{|I_j| } \int_{\widetilde{I}_j} f \, dt
                =  \frac{ 2 }{ 2|I_j| }  
                         \int_{\widetilde{I}_j} f \, dt
                \\
              & \leq \frac{ 1 }{ |\widetilde{I}_j |} \int_{\widetilde{I}_j} f \, dt 
                 \leq 2 \lambda.
                     \end{align*}

                     Regardless of $I_j$, the $\widetilde{I}_j$ we have defined
                     has points from the complement of $G$, and therefore
                     its integral is less or equal than $2 \lambda$.
                     In other words,
                     \begin{equation*}
                     \int_{I_j} f \, dx  \leq 2 \lambda |I_j|
                      \qquad \text{ if and only if }
                            \qquad 
                     \int_{\widetilde{I}_j} f \, dx \leq 2 \lambda |\widetilde{I}_j|
                     \end{equation*}
                     given that $\widetilde{I}_j$ contains at least one point
                     from $F= G^c = \displaystyle \left\{ 
                                 f^\ast \leq \lambda
                                      \right\} .$\\
                                      
Suppose now that we have a Poisson kernel and a function $f$
such that 
          $f\geq 0, $ 
           $\mbox{supp}{(f)}\subset J_k$ 
          and $f$ is \emph{bad},
by which we mean that $f$ is infinite in a dense subset (i.e., for all $n$  there exists $E_n \subset J_k$
 such that  $|f|>n $).  Even though $f$ is bad, we know that for some $k_o$
 \begin{equation*}\label{desi1}
         \frac{1  }{|J_{k_o}|  } \int_{J_{k_o}} f \, dx  \leq \lambda
         \qquad 
         \text{and}
         \qquad 
         d( u, J_{k_o}) \geq c |J_{k_o}|
 \end{equation*}
 \begin{align}
         \int_{J_{k_o}} \frac{ \varepsilon }{ \varepsilon^2+ (u -v)^2 } 
         \ f(v) \ dv & \leq C(c)     \lambda \label{marcin-key}
         \\
         \int_{J_{k_o}} \frac{ \varepsilon }{ \varepsilon^2+ (u -v)^2 } 
                   \ f(v) \ dv & \leq    
            C 
         \int_{J_{k_o}} \frac{ \varepsilon }{ \varepsilon^2+ (u -v)^2 } 
            \ \phi(v) \ dv
 \end{align}
{where $\phi(v)\leq C(c) \lambda$ and $|u-v|> c |J_{k_o}|$.}
 \begin{proof}
         We will prove (\ref{desi1}). Take $v\in J_{k_o} $ such that $ d(u, J_{k_o}) > C |J_{k_o}|,$
         \begin{align*}
         \int_{J_{k_o}} \frac{ \varepsilon }{ \varepsilon^2 + (u -v)^2 } 
                   \ f(v) \ dv 
                  & \leq       
      \int_{J_{k_o}} \frac{ \varepsilon }{ \varepsilon^2 + c^2|J_{k_o}^2|}
                   \ f(v) \ dv 
                   \qquad \text{ if } |u-v|<C |J_{k_o}|
                   \\
                   & \leq
             \frac{ \varepsilon }{ \varepsilon^2 + c^2 |J_{k_o}^2|} 
             \int_{J_{k_o}} f(v) \, dv
                    =  
        \frac{ \varepsilon |J_{k_o}|}{ \varepsilon^2 + c^2|J_{k_o}^2|}
        \ \frac{ 1 }{ |J_{k_o}|} 
             \int_{J_{k_o}} f(v) \, dv
                   \\
                   & \leq
        \frac{ \varepsilon|J_{k_o}|}{ \varepsilon^2 + c^2|J_{k_o}^2|}
          \lambda
                  =  
        \frac{ \varepsilon \lambda}
          { \varepsilon^2 + \frac{ c^2 }{4} \left( 2|J_{k_o}|^2 \right)^2}
           | J_{k_o}| 
                   \\
                   & \leq
        \frac{ \varepsilon \lambda}{ \varepsilon^2 + \frac{ c^2 }{4} \left( |J_{k_o}|^2 +u\right)^2}
         \end{align*}
         
\begin{center}   
\setlength{\unitlength}{4144sp}%
\begingroup\makeatletter\ifx\SetFigFont\undefined%
\gdef\SetFigFont#1#2#3#4#5{%
  \reset@font\fontsize{#1}{#2pt}%
  \fontfamily{#3}\fontseries{#4}\fontshape{#5}%
  \selectfont}%
\fi\endgroup%
\begin{picture}(3849,780)(-11,-391)
{\color[rgb]{0,0,0}\thinlines
\put(2026,-61){\circle*{66}}
}%
{\color[rgb]{0,0,0}\put(901, 74){\line( 0,-1){270}}
}%
{\color[rgb]{0,0,0}\put(1801, 74){\line( 0,-1){270}}
}%
{\color[rgb]{0,0,0}\put(2701, 74){\line( 0,-1){270}}
}%
{\color[rgb]{0,0,0}\put(901,-106){\line( 1, 0){900}}
}%
{\color[rgb]{0,0,0}\put(901,-16){\line( 1, 0){1800}}
}%
{\color[rgb]{0,0,0}\put(  1,-61){\line( 1, 0){3825}}
}%
\put(1261,-376){\makebox(0,0)[lb]{\smash{{\SetFigFont{14}{16.8}{\familydefault}{\mddefault}{\updefault}{\color[rgb]{0,0,0}$|J_{k_0}|$}%
}}}}
\put(1660,109){\makebox(0,0)[lb]{\smash{{\SetFigFont{14}{16.8}{\familydefault}{\mddefault}{\updefault}{\color[rgb]{0,0,0}$2\cdot |J_{k_0}|$}%
}}}}
\put(2026,-241){\makebox(0,0)[lb]{\smash{{\SetFigFont{14}{16.8}{\familydefault}{\mddefault}{\updefault}{\color[rgb]{0,0,0}$v$}%
}}}}
\end{picture}%
\end{center}
Therefore this must hold for the average on $v$ 
         \begin{align*}
         \int_{J_{k_o}} \frac{ \varepsilon }{ \varepsilon^2 + (u -v)^2 } 
                   \ f(v) \ dv 
                  & \leq       
                   \varepsilon \, \lambda\,  
                | J_{k_o}|  \frac{ 1 }{|J_{k_o}} |
                   \int_{J_{k_o}}
        \frac{ dv }
       { \varepsilon^2 + \frac{ c^2 }{4} \left( |I_{k_o}| +v\right)^2}
                   \\
                   & =
                   \int_{J_{k_o}}
        \frac{ \varepsilon }
       { \varepsilon^2 + \frac{ c^2 }{4} \left( |I_{k_o}|+v\right)^2}
       \ \lambda  \ dv
                   = C \, \lambda
         \end{align*}
 \end{proof}
 \eqref{marcin-key}  was the key point in Marcinkiewicz's proof.

\subsection{On the Marcinkiewicz function}

Let $\mathcal{F}(x)$ be the Marcinkiewicz function, defined in (\ref{MarcinFunct}) and  
             $\big\{ I_k\big\} $
             the covering of $G$ satisfying the properties
             of Lemma~\ref{lemacub}
and  consider 
\begin{equation} \label{ser}
    \sum_{k=1}^\infty  
         \big(
                \int_{I_k} \frac{ |I_k| }{ (x-y)^2 }  f_k \ dy
                \big)
    \end{equation}
  where $f_k = \displaystyle \frac{ 1 }{ |I_k | } \int_{I_k} f(t)\ dt$.
  Then it is not difficult to see that if $x\in F= G^c$,
  then (\ref{ser}) is finite and
      \begin{equation*}
      \begin{split}
\int_F \big[
    \sum_{k=1}^\infty  
         \big(
                \int_{I_k} \frac{ |I_k| }{ (x-y)^2 }  f_k \ dy
                \big)
\big]
    & = 
     \sum_{k=1}^\infty \int_{I_k}
         \big[
                 f_k \big( |I_k| \int \frac{ 1 }{ (x-y)^2 } \,dx\big)
               \big] dy \\
        & \leq
                C \sum_{k=1}^\infty  \int_{I_k } f_k \, dy \\
        & \leq
                \lambda C |G|
      \end{split}
      \end{equation*}
where the last inequality follows from the construction of the covering
and the first reflects the fact that 
  $ |I_k| \int \frac{ 1 }{ (x-y)^2 } \ dx \leq C $
  because $| I_k | \int \frac{ 1 }{ x^2 }  \, dx \leq C $ for 
  $|x| \ge |I_k| .$ Furthermore
    \begin{equation*}
    \begin{split}
            \Big| \{ \mathcal{F} (x) > \lambda\}\Big|
            &\leq  \frac{ C }{ \lambda } \sum_k \big(\int f_k (y) \, dy\big)
            \\
            & \leq \frac{C}{ \lambda }  \lambda |G |
            \\
            & \leq C \frac{ 1 }{ \lambda }  \int f(x)\, dx
    \end{split}
    \end{equation*}
  In other words, the Marcinkiewicz function $\mathcal{F}$ is $(1,1)$-weak.
          \begin{proposition}
                  Let $\mu$ be a measure such that $\mu\in A_1$ 
                  and $ d\mu = g \, dx $, where $g$ is the density
                  of the measure $\mu$. Then
                  \begin{equation*}
                     \int_G g \, dx \leq   \frac{C}{\lambda  } \int f\, g\ dx
                  \end{equation*}
          \end{proposition}
\begin{proof}
   Let $\mathcal{F}(x)$ be the Marcinkiewicz function
   \begin{equation*}
   \begin{split}
           \int \mathcal{F}(x) \, d\mu(x)
           &\leq \int_F \, d\mu(x) 
            \int \frac{ 1  }{ (x-y)^2  } |I_k| f_k (y) \, dy
            \\
          & \leq \int_{I_k}  f(y)
                \big[
                        \frac{ d\mu(x) }{[ (y-x) + |I_k|]^2}  |I_k|
                       \big] dy
          \\
          & \leq \mathcal{F} g(y) \leq C \, g(y) \leq C \int f(y) g(y) \, dy
   \end{split}
   \end{equation*}
since  $(x-y)^2\sim \big[(y-x)+|I_k|\big]^2$.
   Therefore using Chebyshev's inequality
   \begin{equation*}
   \Big(\mu\{x: \mathcal{F}(x) > \lambda\}\Big) \leq 
   \frac{ C }{ \lambda } \int_F \mathcal{F}(x) \ d\mu(x)
   \end{equation*}
   over $F$. Then by Marcinkiewicz's theorem
   \begin{equation*}
          \mu(G)  \leq \frac{ C }{ \lambda } \int f(y)g(y)\, dy.
   \end{equation*}
          \end{proof}

   \begin{theorem}[surprising result]
                   \begin{equation}
                           \int   \frac{|f (x+y)|  }{ y^2  } \ dy
                             < \infty  \qquad \text{a.e. in } F
                   \end{equation}
           \end{theorem}
           
\begin{proof}
Let $f_G = f\vert_{_G}$
           and $ G = \bigcup_k J_{_k}$, 
           where $J_{_k}$ denotes maximal intervals, 
           and $\bar{J_k} \cap F \neq \emptyset$.
           We shall see that 
           \begin{equation}
                   \int  \frac{1  }{(x-y)^2} \ f_{_G}(y) \ dy 
                   = 
                   \int \frac{ f_{_G} (x+y)}{ y^2  } \ dy 
                   < \infty \quad  \text {a.e. }
           \end{equation}

           \bigskip
           We shall show next that 
           \begin{equation*}
 \sum_k \int_{J_k}   f_{_G} (y) \ dy \int_F \frac{ 1 }{ (x-y)^2 }  \ dx
           \end{equation*}
           
          Consider $\widetilde{J_k} = (1+\varepsilon) J_k$,
          a dilation by a factor of $1+\varepsilon$
          from the center of $I_k$, therefore
          \begin{equation*}
                  |\widetilde{J_k}| = |(1+\varepsilon) J_k |
                  = (1+\varepsilon) |J_k|
          \end{equation*}
       If $x\in \widetilde{F} = 
                     \Big( \bigcup (1+\varepsilon) J_k  \Big)^c
                     \subset  F$
       then
       \begin{equation*}
               \int_{\widetilde{F}}  \frac{ 1 }{ (x-y)^2 } \ dx
               \leq
               \int_{|x-y|> \varepsilon |J_k|}
                \limits
                  \frac{ 1 }{ (x-y)^2 } \ dy 
                  \leq
                  \frac{ 1 }{ \varepsilon  |J_k|} 
       \end{equation*}
       Then
       \begin{equation*}
               \int_{J_k}   f_{_G} (y) 
               \Bigg(
                      \int_{\widetilde{F}} \frac{ dx }{ (x-y)^2 }  
               \Bigg)\ dy
                      \ \leq \
                  \frac{ 1 }{ \varepsilon  |J_k|} 
                  \int_{J_k}  f_{_G} (y) dy 
                  \ \leq \frac{ 1 }{ \varepsilon } \lambda
       \end{equation*}

 We observe that when $x$ is outside the intervals $I_k$,
         that is  whenever $x \in \widetilde{F}$, the integral
         $ \int_{I_k} \frac{ 1 }{ (x-y)^2 }  f_{_G} (y) dy$
         is bounded from below and from above. 
         If $x\notin I_k$ 
                  $y_k \in I_k$  and if we note that
                  $d(x, I_k) \geq \varepsilon |I_k|$ 
         \begin{equation*}
                 \begin{split}
                 \int_{I_k}  \frac{ 1 }{ (x-y)^2 } f_{_G} (y) \ dy
                &\sim
                   \frac{ 1 }{ (x - y_k)^2 } 
                   \int_{I_k} f_{_G} (y) \ dy
                 \\
                 & = \frac{ |I_k| }{ (x-y_k)^2  } 
                 \frac{ 1  }{  |I_k|  } 
                 \int_{I_k}  f_{_G}(y)\ dy 
                 \\
                 & \leq 
                 \underbrace{
                 \frac{ 1  }{ |I_k| }  \int_{I_k} f_{_G} (y) \ dy
                            }_{= \lambda} \
                             \int_{I_k} \frac{ 1 }{ (x-y)^2 }  \ dy
                 \\
                 & =
                  \lambda 
                   \int_{I_k} \frac{ 1 }{ (x-y)^2 } \ dy
                 \end{split}
         \end{equation*}
         Summing up over $k$ we obtain 
         \begin{equation*}
                 \sum_k \int_{I_k} 
                    \Bigg(
                         \int_{\widetilde{F}} \frac{ dx }{ (x-y)^2 } 
                    \Bigg) \ dy
                    \leq
                   C \lambda \int_{\widetilde{G}}  \frac{ 1 }{ (x-y)^2 } \ dy.\\
         \end{equation*}
\end{proof}

     Now  we shift our attention on a more pure version.
         If $G = \bigcup I_k$, $F=G^c$, $\delta_k= c |I_k|$
         and defining $\phi$ by
         \begin{align}
                         \phi (x)  &=
                    \left\{  
                    \begin{array}{lcl}
                    \displaystyle
                    c \, |I_k|&  \qquad& \text{ if }  x\in I_k   \\[3.5mm]
                          0    &  \qquad& \text{ if }  x\in F     
                    \end{array}
                    \right.
                    \\
                    \intertext{we have}
                    \phi (x)  
                    &=
                    \sum_k c |I_k| \ \chi_{_{I_k}}(x)
         \end{align}

                 \begin{lemma}
                 \begin{equation}
                         \int_{F'} \frac{ 1 }{ (x-y)^2} \ \phi(y) \ dy  
                         \ < \ \infty \qquad \text{ a.e. on } F
                 \end{equation}
         \end{lemma}

         \begin{proof}
               Let $ F'= (G')^c$, where 
                   $ G' = \bigcup_k (1+\varepsilon)I_k$.
                   Denoting by $I'_k = (1+\varepsilon)I_k$
                   \begin{equation}
                      \begin{split}
                         \int_{F'} \int \frac{ 1 }{ (x-y)^2} \phi(y)\ dy\ dx
                         & = 
                         \sum_k \int_{I'_k} 
                                 \sum_k c_k |I_k|
                                  \int_{F'} \frac{ 1 }{ (x-y)^2 }\ dx \ dy \\
                         & = 
                         \sum_k \int_{I_k} c\ dy  \\
                         & = 
                         C'' |G|
                      \end{split}
                   \end{equation}
                  where we have used that  
                   ${ (x-y)^2 }  \geq (1+\varepsilon) |I_k| $
                   if $x\in F'$.  Therefore
                   \begin{equation}
                           \begin{split}
                           \int_{F'}
                           \Bigg( \int \frac{ 1 }{ (x-y)^2 } \ \phi(y)\ dy 
                           \Bigg) dx \ &< \infty
                           \\
                           \int \frac{ 1 }{ (x-y)^2 } \ \phi(y)\ dy 
                           \ & < \infty \quad \text{ a.e. on } F
                           \end{split}
                   \end{equation}

         \end{proof}
         
         We plan to discuss the following chain of inequalities
 \begin{equation} \label{doble-dis}
                        \lambda    \leq \frac{ 1 }{ |I_k|}  \int_{I_k} f(u) \ du \leq  2 \lambda    
\end{equation}
     For $f\geq0$, $u\geq 0$ and $v\geq 0$ 
\begin{align*}
        (1-r^2)
        \int_0^{\tfrac{ \pi }{ 2 } }
        &   \frac{ f(u+x) }{ (1-r)^2  (u+v)^2}  
        \int_0^{\pi/2}
           \frac{ f(v+x) }{ (1-r)^2  (u-v)^2}   \ dv \ du
      \\    
      &\leq  
        (1-r^2)
        \int_0^{\tfrac{ \pi }{ 2 } }
           \frac{ f(u+x) }{ (1-r)^2  u^2}  
        \int_0^{\tfrac{ \pi }{ 2 } }
     \frac{ f(v+x) }{ (1-r)^2  (u-v)^2}   \ dv \ du \quad \text{since $v\geq0$}
     \\
      &\leq  
        (1-r^2)
        \int\limits_x^{x+\tfrac{ \pi }{ 2 } }
        \frac{ f(\bar{u}) }{ (1-r)^2  (\bar{u}-x)^2}  
        \int_0^{\tfrac{ \pi }{ 2 } }
      \frac{ f(\bar{v}) }{ (1-r)^2  (\bar{u}-\bar{v})^2}   \ d\bar{v} \ d\bar{u} 
        \\
      &\leq  
        (1-r^2)
        \int_{\pi}^{-\pi}
           \frac{ f(u) }{ (1-r)^2  (u-x)^2}  
        \int_{-\pi}^{\pi}
           \frac{ f(v) }{ (1-r)^2  (u-v)^2}   \ dv \ du
\intertext{  } 
\end{align*}
Then
\begin{align*}
        (1-r^2)
        \int_{-\pi}^{\pi}
        &   \frac{ f_k(u) }{ (1-r)  (u-x)^2}  
        \sum_{\substack{j=k'\\ xx= k''}}
        \int_{-\pi}^{\pi}
     \frac{ f_j(\cdot) }{ (1-r)  (u-v)^2}   \ du \ dv 
     \\
     &\leq 
        (1-r^2)
        \int_{-\pi}^{\pi}   
        \frac{ f_k(u) }{ (1-r)^2  (u-x)^2}  
        \sum_{\substack{j=k'\\ xx= k''}}
     \frac{ f_j(\cdot) }{ (1-r)^2  }   \ dv \ du 
     \\
     &\leq 
        \int_{-\pi}^{\pi}   
        \frac{ f_k(u) }{ (u-x)^2  }  C \lambda |I_k| \;  \ du 
\end{align*}
Let
\begin{align*}
       f_k &= \left\{  
       \begin{array}{lcl}
       \displaystyle  f 
         &  \qquad& \text{on $I_k$} \  \\[3mm]
           0            &       &  \text{elsewhere.}
       \end{array}
       \right.
\intertext{If $ x \in I_j$, then }
      \varphi_k (x)  
         &=  \int_{\chi_{_k}}  
               \frac{ 1-r }{ (x-y)^2+(1-r)^2  } f_k(y) \ dy
               \\
         &\leq  \int_{\chi_{_k}}  
               \frac{ 1+r }{ (1-r)^2 + c(x-y)^2 }\ \mu_k \ dy
\end{align*}

\begin{align*}
        \int  \frac{1-r  }{(1-r)^2 + (x-y)^2  }  \ f_k(y) \ dy 
        &\leq
        \big(
                \frac{1-r  }{(1-r)^2 + (x-y)^2  }   \ dy 
              \big)  \ \mu_k
               \\
  &\leq  \big(
                \frac{1-r  }{(1-r)^2 + C (x-y)^2  }   \ dy 
       \big)\ \mu_k
\end{align*}

\begin{align*}
        \int  \frac{ \varepsilon^{p-1} }{ \varepsilon^p + (x-y)^p  }  \ dy
         &=      
        \int  \frac{ \varepsilon^{p-1} } {\varepsilon^p + |y|^p} \ dy
       =      
        \int  \frac{ \varepsilon^{p-1} } 
              { 1 + \big|\frac{ y }{\varepsilon  } \big|^p } \ dy
              \\
         &=      
        \int  \frac{ \varepsilon^{-1} } 
              { 1 + \big| \frac{ y }{\varepsilon  }\big|^p} \ dy
      =      
     \int \frac{ ds }{ 1 + |s|^p },  \qquad \frac{ y }{ \varepsilon } = s.
\end{align*}

\subsection{Power Series}
Now, consider the power series, $\sum_{k=0}^\infty a_k x^k $,
for $x$ a complex variable, which is convergent for $|x|<1$
and let $ f(x) = \sum_{k=0}^\infty a_k x^k $. Then $f$ is analytic.

  Denoting the partial sum of  $\sum_{k=0}^\infty a_k$ by $S_n = \sum_{k=0}^n a_k $, then, as $S_k - S_{k-1} =a_k$, then
  \begin{equation}
        \begin{split}
               f(x) & = \sum_{k=0}^\infty  \left(  S_k - S_{k-1}\right) x^k  = \sum_{k=0}^\infty S_k x^k - \sum_{k=1}^\infty S_{k-1} x^{k} \\
                    & = \sum_{k=0}^\infty S_k x^k - \sum S_{k} x^{k+1} = \sum_{k=0}^\infty S_k (x^{k+1} - x^k)  = ( 1-x ) \sum_{k=0}^\infty S_k x^k
        \end{split}
\end{equation}
Suppose now  $x \in B_1 = \big\{ x  :  |x| < 1\big\}  $, then
\begin{align}
        \frac{f(x)  }{ 1 -x   }  &= \sum_{k=0}^\infty S_k x^k \\
 \intertext{  Writing $ x^k = r^k e^{ik\theta }$,   }   
 \frac{  f(x)  }{ 1 -x   }  &= \sum_{k=0}^\infty \left(S_k r^k \right) e^{ik\theta}\\
\intertext{If $1<p<2$, then using the Hausdorff-Young inequality we have}
\left[  \sum_{k=0}^\infty \big| S_k r^k \big|^{q}
\right]^{1/q}
        &\leq
        C_p \left(
        \int_0^{2\pi} \big|  
                         \frac{ f(r e^{i\theta}) }{1-re^{i\theta}  \big|
                            }
           d\theta
           \right)^{1/p}\\
  \nonumber& \quad \quad \times      \big[ \int_0^{2\pi} \frac{ 1 }{ 1-re^{i\theta} } 
            \big( \int_0^{2\pi} 
                         |P(r, \theta-\psi )| f(\psi)\ d\psi
                       \big)^p
               \big]^{1/p}
\end{align}

Now, 
\begin{align*}
    \sum_{k=0}^\infty a_k e^{i k s } e^{ik\theta } r^k 
        &= f \big( r e^{i(\theta+ s)}\big)
        \\
        f(r, \theta+s) & = \sum_{k=0}^\infty (1-z) S_k z^k 
        \qquad \text{ where } z=e^{is}r  \text{ and }S_n =\sum_{k=0}^n a_k e^{ik\theta}.
\end{align*}
Now $\theta$ has been fixed and $s$ varies, therefore
\begin{equation*}
       \sum_{k=0}^\infty S_k z^k =  \frac{1  }{1-z}   f(r, \theta+s).
\end{equation*}
By the Hausdorff-Young inequality which  bounds  the $L^p$ norm 
the Fourier coefficients  for $ 1 < p < 2 $ and $2 < q < \infty$
\begin{align}
        \big( \sum_{k=0}^\infty r^k q |S_k|^q \big)^{1/q} \label{bound}
        &\leq C_p 
         \big(\int_0^{2\pi} \frac{ 1 }{ |1 - re^{is}|^p} f(r,\theta+s)^p 
        \big)^{1/p}.
        \big)
\intertext{Now, since we have the equivalence 
        $ \big| 1 - r e^{is}\big| \sim \big|(1-r)^2 + s^2\big|^{1/2} $
        in the sense it is bounded by \eqref{bound} multiplied by a constant } 
\nonumber          &\leq \big(
          \int\limits_0^{2\pi} \frac{ 1 }{ \big[ (1-r)^2+s^2\big]^{p/2} }
          \big( \int P (r, \theta+s - u ) p(u) \ du\big)^p
                      \big)^{1/p}
      \\
 \nonumber      &=\big(
          \int\limits_0^{2\pi} \frac{ 1 }{ \big[ (1-r)^2+(s-\theta)^2\big]^{p/2} }
          \big( \int P (r, \theta+s - u ) p(u) \ du \big)^p
                      \big)^{1/p}
\end{align}

 Using a Whitney type decomposition
       with $f = \sum_{k=0}^\infty f_k$ we can rewrite the previous equation as
\begin{equation*}
      \big(   \int\limits_0^{2\pi}  
             \frac{ 1 }{ \big[ (1-r)^2+(s-\theta)^2\big]^{p/2}}
             \big(
                     P (r, s-u) f_k(u) \sum_j 
                     \big( \int P (r, s-v) f_j(v) \ dv \big)\big)
                 \big)^{p-1}
\end{equation*}
{where $1-r = \varepsilon$ and where for now we shall work with 
  a single $k$ to bound $f_k$ and then we shall consider all $k$}
\begin{equation*}
        \label{cases}
\big( 
         \int_0^{2\pi}  
             \frac{ 1 }{ \big[ \varepsilon^2+(s-\theta)^2\big]^{p/2}}
            \big(
                     P (r, s-u) f_k(u) \sum_j 
                     \big(\int P (r, s-v) f_j(v) \ dv \big)
                  \big)
         \big)^{p-1}
\end{equation*}

We consider the different possible cases separately:
\begin{enumerate}[leftmargin=1.7cm, rightmargin=0cm, label=\underline{Case \alph*:} ]
        \item 
                We start with case when the $j$ are not adjacent to $I_k$.
  \begin{align*}
     \eqref{cases} 
     &\leq
     \big( 
             \int_0^{2\pi} 
             \frac{ 1 }{ \big[ \varepsilon^2+(s-\theta)^2\big]^{p/2} }
             \big(  \int  P (r, s-v)  \lambda \chi_{_k} C \lambda^{p-1}\big)
             \big)
\intertext{considering $s$ not in $I_k$}
    &\leq
  \big( 
             \int_0^{2\pi} 
             \frac{ 1 }{ \big[ \varepsilon^2+(s-\theta)^2\big]^{p/2}}
            \big(  \int  P (r, s-v)   C \lambda \lambda^{p-1} \big)
            \big)
  \end{align*}
  \item
          In the case that $j$ touch the adjacent,
          the measures are comparable and we have
  \begin{align*}
     \eqref{cases} 
     &\leq
    \big( 
             \int_0^{2\pi} 
             \frac{ 1 }{ \big[ \varepsilon^2+(s-\theta)^2\big]^{p/2} }
        \big(  \int  P (r, s-u) C\lambda 
                        \big(\lambda |I_k|\big)^{p-1} \big)
 \big)
 \big)
     \intertext{In the case that $s$ not in $I_k$}
    &\leq
\big( 
             \int_0^{2\pi} 
             \frac{ 1 }{ \big[ \varepsilon^2+(s-\theta)^2\big]^{p/2}}
             \big(  \int_{I_k}  P (r, s-u) \lambda \chi_{_k} (u) C
                        \big(\lambda |I_k|\big)^{p-1} \big)
                        \big)
     \\
    &\leq
    \big( 
             \int_0^{2\pi} 
             \frac{ 1 }{ \big[ \varepsilon^2+(s-\theta)^2\big]^{p/2} } 
             \underbrace
             { \big(  \int_{I_k}  P (r, s-u) \lambda^p \chi_{_k} (u) C
                    \big)}_{C\lambda^p}
  \big)
  \end{align*}
\end{enumerate}

\subsubsection{Abel Sums analog}

Consider the series $(1-r)\sum_{\nu=0}^\infty r^\nu (S_\nu(f,x))^2$, then taking $r=1- \frac{1}{n}$ then
\begin{eqnarray*}
(1-r)\sum_{\nu=0}^\infty r^\nu (S_\nu(f,x))^2 &=& (1-(1- \frac{1}{n}))\sum_{\nu=0}^\infty (1-\frac{1}{\nu}) ^\nu (S_\nu(f,x))^2\\
&=& \frac{1}{n}\sum_{\nu=0}^\infty (1-\frac{1}{\nu}) ^\nu (S_\nu(f,x))^2\\
&\geq& e^{-1} \frac{1}{n}\sum_{\nu=0}^\infty  (S_\nu(f,x))^2.
\end{eqnarray*}
Therefore,
$$ (S^\ast f)(x) \leq e^{1/2} \sup_{0<r<1} [ (1-r) (1-r)\sum_{\nu=0}^\infty r^\nu (S_\nu(f,x))^2]^{1/2}.$$
The key will be to study $ (1-r)\sum_{\nu=0}^\infty r^\nu (S_\nu(f,x))^2$. We will construct a kernel
\begin{equation}
D(r, x,y) = \sum_{\nu=0}^\infty r^\nu D_\nu(x) D_\nu(y),
\end{equation}
where $D_\nu$ is the Dirichlet kernel 
$$ D_\nu (x) = \frac{\sin(\nu +1/2)x}{2 \sin (x/2)}.$$
Thus,
$$ D(r, x,y) = \frac{(1-r)[(1-r)^2 + 2 r (2+\cos x + \cos y)]}{4(1-2 r \cos(x-y) +r^2)(1-2 r \cos(x+y) +r^2)}.$$
For $-\pi+\varepsilon < (x+y) <\pi-\varepsilon$ and $-\pi+\varepsilon < (x-y) <\pi-\varepsilon$ we get,
$$ D(r, x,y)  \leq \frac{9(1-r)}{[(1-r)^2 +rC_3(x-y)^2][(1-r)^2 +rC_3(x+y)^2]}.$$
We will assume $f=0$ if $|x| > \pi/2 - \varepsilon /2.$ and we will estimate
$$(1-r)\sum_{\nu=0}^\infty r^\nu (S_\nu(f,x))^2 =\frac{1-r}{\pi} \int_{-\pi}^\pi \int_{-\pi}^\pi  f(x+u)g(x+v) D(x,r,u,v) du dv,$$
We move from the periodic case to the continuous case. The above integral is dominated by
$$ (1-r)^2 \int_{-\infty}^\infty \int_{-\infty}^\infty  \frac{9f(x+u)g(x+v) } {[(1-r)^2 +rc(u-v)^2][(1-r)^2 +r c(u+v)^2]} du dv.$$
We shall consider the integral 
\begin{equation}\label{case1}
(1-r)^2 \int_{0}^\infty \int_{0}^\infty  \frac{f(x+u)g(x+v) } {[(1-r)^2 +rc(u-v)^2][(1-r)^2 +r c(u+v)^2]} du dv,
\end{equation}
and the analogous integrals on the regions $(-\infty,0)\times (-\infty,0)$, $(-\infty,0)\times (0,\infty)$ and $(0,\infty)\times (-\infty, 0)$, that can be argue as the preceding one using symmetry arguments.\\
Let us study the integral (\ref{case1}), which is typical
\begin{eqnarray*}
&&(1-r)^2 \int_{0}^\infty \int_{0}^\infty  \frac{f(x+u)f(x+v) } {[(1-r)^2 +r c(u-v)^2][(1-r)^2 +r c(u+v)^2]} du dv\\
&\leq& (1-r)^2 \int_{0}^\infty \int_{0}^\infty  \frac{f(x+u)f(x+v) } {[(1-r)^2 +r c(u-v)^2][(1-r)^2 +r cv^2]} du dv
\end{eqnarray*}
$f$ has already been decomposed as $f= \sum_{k=0}^\infty f_k$ where 
$$f_k = \begin{cases} f &\mbox{on}\; I_k \\
0  &\mbox{otherwise.}
\end{cases},$$
$\{I_k\}$ are such that
$$ \frac{1}{|I_k|} \int_{I_k} f(t) dt \leq C \lambda.$$
The above integral (\ref{case1}) is dominated by
$$ C (1-r)^2 \int_{-\infty}^\infty \int_{-\infty}^\infty  \frac{f(x+u) } {[(1-r)+rc(u-v)^2]}\frac{f(x+v)}{[r v^2]} du dv,$$
and after a change of variables, we have,
$$ \frac{c}{r} \int_{-\infty}^\infty  \int_{-\infty}^\infty \frac{f(v)}{(v-x)^2} \frac{(1-r)^2 f(u)}{(1-r)} du dv.$$
The intervals $I_k$ have been constructed so that
\begin{enumerate}
\item [i)] If $I_j$ and $I_k$ are adjacent then
$$\frac{1}{2} |I_j| \leq |I_k| \leq 2 |I_j|.$$
\item [ii)] $d(I_j,I_k) \geq \frac{1}{2} |I_j|,$ and $d(I_j,I_k) \geq \frac{1}{2} |I_k|$
\end{enumerate}
For $x \in F= \bigcup_\nu I_\nu$ decompose the double integral as the sum
$$ \sum_{i,j} \int_{I_i}\int_{I_j} (1-r)^2 \int_{I_i} \frac{f(v)}{(1-r)^2+c(v-x)^2}  dv \int_{I_j} \frac{f(u)}{(1-r)^2+c(u-v)^2} du $$
\begin{enumerate}
\item [i)] adjacent $I_j$: there are at most 2 of those intervals,
$$ (1-r)^2 \int_{I_i} \frac{f(v)}{(1-r)^2+c(v-x)^2}  dv \int_{I_j} \frac{f(u)}{(1-r)^2+c(u-v)^2} du  \leq \int_{I_i} \frac{f(v)}{(v-x)^2}  dv  \times C \lambda |I_i| \leq C \lambda \mathcal{F}(x).$$
\item [ii)] non adjacent $I_j$:
$$ (1-r)^2 \int_{I_i} \frac{f(v)}{(1-r)^2+c(v-x)^2}  dv \times \sum_{j} C \int_{I_j} \frac{\lambda}{(1-r)^2+(v-x)^2)} \leq C \lambda^2.$$
\end{enumerate}

\end{document}